\documentclass{amsart}
\usepackage{amsmath,amsfonts,amssymb,amsthm,hyperref}

\usepackage{color, tikz,enumitem}
\usepackage{bm}

\numberwithin{equation}{section}
\newtheorem{claim}{claim}[section]
\newtheorem{theorem}[claim]{Theorem}

\newtheorem{lemma}[claim]{Lemma}
\newtheorem{proposition}[claim]{Proposition}

\newtheorem{corollary}[claim]{Corollary}
\newtheorem{conjecture}[claim]{Conjecture}

\newtheorem{remark}[claim]{Remark}

\let\infp\relax \DeclareMathOperator*\infp{\vphantom{p}inf}

\newcommand{\C}{\mathbb{C}}

\newcommand{\ds}{\Omega_n}

\newcommand{\Pn}{\mathcal{P}_n}

\newcommand{\conv}{\text{Conv}}

\newcommand{\F}{{\operatorname{F}}}

\newcommand{\lp}[1][p]{{\ell_n^{#1}\to\ell_n^{#1}}}

\author{Ludovick Bouthat, Javad Mashreghi and Frédéric Morneau-Guérin}

  \title{On the Geometry of the Birkhoff Polytope \\ {I. The operator \texorpdfstring{$\bm{\ell^p_n}$}{ℓp}-norms}}

\begin{document}

\begin{abstract}
{ The geometry of the Birkhoff polytope, i.e., the compact convex set of all $n \times n$ doubly stochastic matrices,  has been an active subject of research. While its faces, edges and facets as well as its volume have been intensely studied, other geometric characteristics such as the center and radius were left off, despite their natural uses in some areas of mathematics. In this paper, we completely characterize the Chebyshev center and the Chebyshev radius of the Birkhoff polytope associated with the metrics induced by the operator $\ell^p_n$-norms for the range $1 \leq p \leq \infty$.}
\end{abstract}


\maketitle


\section{Introduction}

A square matrix $D= [d_{ij}]$ is said to be \textit{doubly stochastic} if every entry of $D$ is non-negative and if each row and each column of $D$ sums up to 1, i.e.,
$$
d_{ij} \geq 0 ~~\quad\&~~\quad \sum\limits_{i=1}^n d_{ij}= \sum\limits_{j=1}^n d_{ij}=1,
$$
for all $i, j = 1,2, \dots, n$. Doubly stochastic matrices appear naturally in many different mathematical contexts such as the theory of majorization \cite{Marshall2011}.

Let $\ds$ denote the set of $n\times n$ doubly stochastic matrices. It is well known that $\ds$ is a semigroup with respect to matrix multiplication and that it is a convex polytope (i.e., a compact convex set with a finite number of extreme points) in the Euclidean space of dimension $n^2$. Moreover, it was shown by Birkhoff \cite{Birkhoff1946} that the extreme points of $\ds$ are precisely the $n\times n$ permutation matrices. We write $\ds=\conv(\Pn)$, where  $\Pn$ denotes the set of $n\times n$ permutation matrices and $\conv(\cdot)$ designates the convex hull of the set that is operated on. More specifically, each $D\in\ds$ admits a (not necessarily unique) \emph{Birkhoff decomposition} $D=\sum_{i=1}^r \alpha_i P_i$, where $P_i\in\Pn$, $\alpha_i \geq 0$, and $\sum_{i=1}^r \alpha_i=1$. Due to this characterization, $\ds$ is sometimes referred to as the \emph{Birkhoff polytope}. 

The geometry of the Birkhoff polytope has been an active subject of research for more than half a century. For instance, in 1977, in a series of four papers, Brualdi and Gibson \cite{BrualdiGibson4,BrualdiGibson1,BrualdiGibson2,BrualdiGibson3} studied the Euclidean geometry structure of $\ds$. In particular, they described the faces, the edges and the facets of $\ds$. In 1996, Billera and Sarangarajan \cite{BilleraSarangarajan1996} pursued this line of study, while also considering two other related polytopes. Then, from 1999 up to 2016, several articles studying the volume of $\ds$ were published. In particular, formulas for the volume of $\ds$ were given by Sturmfels \cite{Sturmfels1997} for $n\leq7$,  by Chan and Robbins \cite{ChanRobbins1999} for $n=8$, and by Beck and Pixton \cite{BeckPixton2003} for $n=9,10$. As for the case of $10\leq n \leq 15$, various estimates were obtained in 2014 by Emiris and Fisikopoulos \cite{EmirisFisikopoulos2014}, and in 2016 by Cousins and Vempala. \cite{CousinsVempala2016}. Meanwhile, De Loera, Liu and Yoshida \cite{DeLoeraLiuYoshida2009} provided an explicit combinatorial formula for the volume of $\ds$ in 2009.


The present series of articles is in line with above-mentioned papers in that the problems addressed are centered around studying the geometry of the Birkhoff polytope. Our main point of interest are the Chebyshev center and the Chebyshev radius of $\ds$. Along the way, we ponder about classical problems such as finding the radius of a minimal bounding ball for the Birkhoff polytope and the smallest enclosing ball problem. Clearly, these notions heavily depend on the metric with which $\ds$ is equipped. As we shall see, there are some advantages of considering \emph{permutation-invariant} submultiplicative norm on $\ds$. As such, we shall consider in the present article the case of the \emph{operator norms from $\ell^p_n$ to $\ell^p_n$} ($1 \leq p \leq \infty$). In the second paper of this series, the case of the \emph{Schatten $p$-norms} ($1\leq p <\infty$) is addressed. 

This paper is structured as follows. In Section \ref{sec - def}, we establish the preliminaries needed for later use. More precisely, we recall some standard definitions in Sections \ref{subsec - def} and \ref{subsec - prop}, we present a few elementary results on the spectrum of doubly stochastic matrices in Section \ref{sec - prop}. In Section \ref{Sec: Remarks}, we give a summary of the main results contained in the present paper. In Section \ref{sec - size}, we derive values for the minimum and maximum distance of an element of the Birkhoff polytope from the origin. In Section \ref{sec - ball}, we study the minimal bounding ball of the Birkhoff polytope. Finally, in Section \ref{sec - Chebyshev}, we study the Chebyshev center and the Chebyshev radius of $\ds$, both in a general setting and in the case of the operator norms from $\ell^p_n$ to $\ell^p_n$ ($1 \leq p \leq \infty$).


\section{Definitions, Properties and Preliminary results} \label{sec - def}


\subsection{Basic definitions}
\label{subsec - def}

In this section, we first recall the definition of some standard notions that play a major role in what follows. Secondly, we state some properties about doubly stochastic matrices, which we use repeatedly in the various demonstrations throughout this article. Finally, we look at some of the remarkable features of a particular doubly stochastic matrix that are used extensively in the rest of the paper.

\subsubsection*{Vector $p$-norms}

Recall that for $p \geq 1$, the \textit{$p$-norm} of a given vector ${x = (x_1, \dots, x_n)}$ is defined by
\[
\|x\|_p \,=\, \left( \sum_{k=1}^{n} |x_k|^p \right)^{\!\!1/p},
\]
and the $\infty$-norm by
\[
\|x\|_\infty \,=\, \max\{ |x_1|, \dots, |x_n|\}.
\]
When $\mathbb{C}^n$ is equipped with the vector $p$-norm, it is customary to denote it by $\ell^p_n(\mathbb{C})$, or by $\ell^p_n$ for short.

\subsubsection*{Operator norms induced by the vector $p$-norms}

Any $n \times n$ matrix $A$ can be interpreted as an operator from $\ell^p_n$ to $\ell^p_n$, in which case its operator norm is given by
\[
\|A\|_{\lp} \,:=\, \sup_{x\neq 0}{\frac {\|Ax\|_p}{\|x\|_p}}.
\]
We recall that for all $1\leq p\leq \infty$, we have $\|B\|_{\lp} = \|B^*\|_{\lp[q]}$, where $*$ denotes the Hermitian conjugate and $q$ is the H\"older conjugate of $p$, i.e.,  $\tfrac{1}{p} + \tfrac{1}{q} = 1$  \cite[p.\,357]{HornJohnson2013}.

Since the sets $\ds$ and $\Pn$ are invariant under the conjugation action, many operators considered below have the same norm, whether we see them as $\ell^p_n \to \ell^p_n$ or as $\ell^q_n \to \ell^q_n$ mappings. Consequently, when such situation arises, for the sake of brevity we shall only consider the case $1\leq p \leq 2$.







\subsection{Useful properties of norms}
\label{subsec - prop}

We now turn our attention to two specific properties that norms on a given vector space of matrices may possess.

\subsubsection*{Submultiplicativity}

An important feature distinguishing matrices from rearranged vectors is the matrix multiplication. The way in which a given norm behaves with respect to matrix multiplication is therefore of particular interest.

A norm $\|\cdot\| : M_{n}(\C) \rightarrow [0,\infty)$ is said to be \textit{submultiplicative} if, for all $A, B \in M_{n}(\C)$,
\[
\|AB\| \,\leq\, \|A\|\|B\|.
\]
Many authors reserve the terminology \textit{matrix norm} for submultiplicative norms.

Note that both the matrix norms induced by vector $p$-norms $(1 \leq p \leq \infty)$ and the Schatten $p$-norms $(p \geq 1$) are submultiplicative. As for an example of a norm on $M_{n}(\C)$ that is not submultiplicative, consider the maximum norm $\|A\|_{\max} := \max\limits_{i,j} |a_{ij}|$.

\subsubsection*{Permutation-invariant norms}\label{PINdef}

A norm $\|\cdot\| : M_{n}(\C) \rightarrow [0,\infty)$ is \textit{permutation-invariant} if for any permutation matrices $P, Q \in \Pn$ and any matrix $A \in M_{n}(\C)$, we have
\[
\|PAQ\| \,=\, \|A\|.
\]

It should be underlined that, the matrix norms induced by vector $p$-norms  $(1 \leq p \leq \infty)$, the Schatten $p$-norms $(p \geq 1)$ and the maximum norm are all permutation-invariant.


\subsection{Elementary properties of doubly stochastic matrices}
\label{sec - prop}

Let us recall two classical results on the theory of doubly stochastic matrices which are needed in our studies.
\begin{enumerate}[label=(\roman*)] 
    \item A doubly stochastic matrix always has the eigenvalue 1 corresponding to the eigenvector $e=(1,1,\dots,1)^\intercal$. All other eigenvalues are in absolute value smaller or equal to 1 \cite[Lemma 1]{Perfect1965}.
    %
    %
    \item\label{prop - conv-max} A convex real-valued function on $\Omega_n$ attains its maximum at a permutation matrix   \cite[Corollary 8.7.4]{HornJohnson2013}.
\end{enumerate}


\subsection{A special doubly stochastic matrix}
\label{subsec - assign}

The $n \times n$ matrix where every entry is equal to $1/n$, which we denoted by $J_n$, plays an important role in the whole theory.  This doubly stochastic matrix is special in a number of regards. Firstly, it acts as \textit{the absorbing element} for $\ds$. That is to say $DJ_n = J_n D = J_n$ for every $n \times n$ doubly stochastic matrix $D$. Secondly, as expressed by the following lemma, it is the \textit{isobarycenter} of $\Pn$.

\begin{lemma}\label{lem - bonus3}
The matrix $J_n$ is the uniform convex combination of all the $n \times n$ permutation matrices, i.e.,
\[
J_n\,=\,\frac{1}{n!} \sum_{P\in\Pn} P.
\]
\end{lemma}

\begin{proof}
Let $D$ be the doubly stochastic matrix given by the convex combination $\frac{1}{n!} \sum_{P\in \Pn} \!P$. Since the permutation group $\Pn$ is invariant under permutations on both sides (viz. $Q\Pn R=\Pn$ for all $Q$ and $R\in\Pn$),  for all $Q, R \in \Pn$, we have
\[
QDR \,=\, \frac{1}{n!} \sum_{P\in\Pn} QPR \,=\, \frac{1}{n!} \sum_{S\in\Pn} S \,=\, D.
\]
This means that $D$ is invariant with respect to the permutation of its rows and columns.  Hence, all rows (resp. columns) of $D$ are identical. Since $D$ is doubly stochastic, we conclude every entry of $D$ is equal to $1/n$.
\end{proof}

\section{Main results}\label{Sec: Remarks}

In this section, we summarize the main results contained in the present work. Note that the geometric notions that appear in this summary will be defined in due course.

The first main theorem concerns the Chebyshev center and radius of $\ds$ relative to a permutation-invariant norm on $M_n(\mathbb{R})$.

\begin{theorem}
    Let $\mathcal{R} \subseteq M_n(\mathbb{R})$ be a convex permutation-invariant constraint set and let $\|\cdot\|$ be a permutation-invariant norm on $M_n(\mathbb{R})$. If there exist a Chebyshev center $A$ of $\ds$ relative to the metric space $(M_n(\mathbb{R}),\|\cdot\|)$ and the constraint set $\mathcal{R}$, then the matrix $J_nAJ_n =  \big(\frac{1}{n}\sum_{i,j=1}^n a_{ij} \big)J_n$ is also a Chebyshev center of $\ds$ relative to the aforementioned metric space and constraint set. Moreover, the Chebyshev radius of $\ds$ in this setting is given by
    \[
    R_{\|\cdot \|}(\ds) \,=\,\|J_nAJ_n-I_n\| \,=\, \inf_{\!\!\substack{\alpha\in\mathbb{R}\\ \alpha J_n \in \mathcal{R}}\!}\|\alpha J_n-I_n\|
    \]
    and the infimum is attained by $\alpha = \frac{1}{n}\sum_{i,j=1}^n a_{ij}$.
\end{theorem}

The second main theorem is in fact a particular case of the previous one, the main difference being that the following is in the particular case where the constraint set is precisely the set of $n\times n$ doubly stochastic matrices.

\begin{corollary}
    If $\|\cdot\|$ is a permutation-invariant norm on $M_n(\mathbb{R})$, then the special doubly stochastic matrix $J_n$ is a Chebyshev center of $\ds$ relative to the metric space $ (\ds,\|\cdot\|)$. Moreover, the associated Chebyshev radius is given by $R_{\|\cdot\|}(\ds)=\|J_n-I_n\|$.
\end{corollary}

The final main theorem regroup many results from this paper and establishes some geometric properties of the Birkhoff polytope $\ds$ within the ambiant space $M_n(\mathbb{R})$ equipped with the metric induced by the matrix norms $\|\cdot\|_{\lp}$ for $1 \leq p \leq \infty$.

\begin{theorem}
	Given $D\in\ds$ and $1 \leq p \leq \infty$, the following hold true: 
	\begin{enumerate}[label=(\roman*)]
		\item $\|D\|_{\lp}=1$;
		\item The special doubly stochastic matrix $J_n$ is the unique Chebyshev center of $\ds$ relative to the metric space $(\ds,\|\cdot\|_{\lp})$.
		\item The Chebyshev radius of $\ds$ relative to the metric space $(\ds,\|\cdot\|_{\lp})$ is given by $\|J_n-I_n\|_{\lp}$, where $I_n$ is the identity matrix. In particular, it is equal to $1$ for $p = 2$ and to $2(1-\tfrac{1}{n})$ for $p=1, \infty$.
	\end{enumerate}
\end{theorem}




\section{The minimum and maximum distance of an element of the Birkhoff polytope from the origin}
\label{sec - size}

We begin our discussion by presenting a useful result on the range of the operator norms from $\ell^p_n$ to $\ell^p_n$ $(1 \leq p \leq \infty)$ when their operand runs through the Birkhoff polytope $\ds$.  Finding lower and upper bounds for this norms (as well as identifying the conditions under which these are optimal) will prove valuable later on. 
%
%
%
%
Since $DJ_n=J_nD=J_n$ for any doubly stochastic matrix $D\in\ds$, then, regardless of our choice of matrix norm,
\begin{equation}\label{eq - lower}
	\|J_n\| \,=\, \|DJ_n\| \,\leq\, \|D\|\|J_n\|
\end{equation}
and thus $\|D\|\geq 1$. We will now show that if $\|\cdot\|$ is the operator norm from $\ell^p_n$ to $\ell^p_n$, then this trivial lower bound is in fact optimal.

\begin{proposition}\label{thm - norm_lp}
    Given $D\in\ds$ and $1 \leq p \leq \infty$, then $\|D\|_{\lp}=1$.
\end{proposition}
\begin{proof}
It suffices to show that $\|D\|_{\lp} \leq 1$. Let $D=\sum_{i=1}^r \alpha_i P_i$ be a Birkhoff decomposition of $D$. 
Since the operator norm from $\ell^p_n$ to $\ell^p_n$ is permutation-invariant, we have
\begin{align*}
    \|D\|_{\lp} \,&=\, \bigg\| \sum_{i=1}^r \alpha_i P_i \bigg\|_{\lp} \leq\, \sum_{i=1}^r \alpha_i \|P_i\|_{\lp} \\
    &=\, \sum_{i=1}^r \alpha_i \|I_n\|_{\lp} \,=\,  \|I_n\|_{\lp}  \,=\, 1. \qedhere
\end{align*}
\end{proof}

\begin{remark}
The same line of ideas can be applied to any permutation-invariant norm $\|\cdot\|$. Thus, we obtain that $\max_{D\in\ds} \|D\| = \|I_n\|= \|P\|$, where $P$ is any permutation matrix $P\in\Pn$.
\end{remark}

\section{The Minimal Bounding Ball of the Birkhoff polytope}
\label{sec - ball}

Given a metric space $(\mathcal{U},d)$, let $\mathcal{B}$ be a nonempty closed, bounded subset of $\mathcal{U}$ and let $B(x,r):= \{ y \in \mathcal{U} : d(x,y) \leq r \}$ denote the closed ball of radius $r>0$ centered at $x\in \mathcal{U}$. We say that $B(x,r)$ is a \emph{bounding ball} of $\mathcal{B}$ centered at $x$ if $\mathcal{B} \subseteq B(x,r)$. The smallest radius $r$ such that $B(x,r)$ is a bounding ball of $\mathcal{B}$ is referred to as the \emph{minimal radius of a bounding ball of $\mathcal{B}$ centered at $x$}. Depending on which information is granted in the context and which is not, it can be either denoted by $r_x(\mathcal{B})$ (for instance see \cite{Goebel1990}) or by $r_d(x)$.

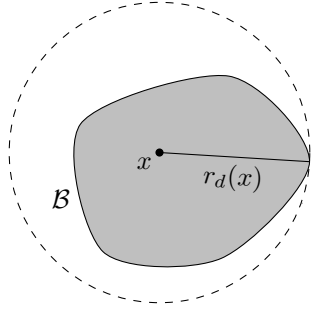
\begin{figure}[h]
	\pgfmathsetseed{15946}
	\centering
	\begin{tikzpicture}[scale=1.35]
		\filldraw[fill=lightgray] plot [smooth cycle, samples=5,domain={1:5}] (\x*360/5+5*rnd:0.5cm+1cm*rnd) node at (0,0) {};
		\fill(1,-0.4)node[below,xshift=-3.2cm,yshift=0.3cm]{$\mathcal{B}$} circle (0cm);
		\draw[dashed](-0.4,0.1) circle (1.48);
		\fill(-0.4,0.1)node[below,xshift=-.2cm,yshift=.05cm]{$x$} circle (.04cm);

		\draw(-0.4,0.1)--(1.075,0.01) node[midway,sloped,yshift=-0.28cm]{$r_d(x)$};
	\end{tikzpicture}
	\caption{The smallest enclosing ball of the nonempty closed bounded set $\mathcal{B}$ centered at $x$ with respect to the metric $d$.}
	\label{fig - Psi}
\end{figure}  

In this series of paper, the ambiant space shall always be $M_n(\mathbb{R})$ while the nonempty closed, bounded set $\mathcal{B}$ of interest shall always be the Birkhoff polytope $\ds$. But the metric (induced by a norm) with which the ambiant space is endowed shall not always be the same. We will thus adopt the notation $r_{\|\cdot\|}(x)$, which places greater emphasis on the metric chosen. Using this notation, we have
\begin{align}\label{def: r(x)}
    r_{\|\cdot\|}(A) \,=\, \sup_{B\in \ds} \!\|A-B\|.
\end{align}
As a matter of fact, the set $\ds$ being compact, the supremum in \eqref{def: r(x)} can be replaced by a maximum, and it follows that the minimal radius of a bounding ball of $\ds$ centered at $A$ exists and is attained for every $A \in M_n(\mathbb{R})$.

We now turn our attention to the problem of characterizing the minimal radius of an enclosing ball of $\ds$ centered at $A\in M_n(\mathbb{R})$ when the ambiant space is equipped with the operator norms from $\ell^p_n$ to $\ell^p_n$. 
%
%
%
%
This problem, for a generic $1 \leq p \leq \infty$, has proven to be quite difficult. Unsurprisingly, the particular case where $p =1$ (and by the same token the case $p= \infty$ since $\|A\|_{\lp[1]} = \|A^*\|_{\lp[\infty]}$) turns out to be more tractable. Recall that in this special case, we have the convenient formula $\|A\|_{\lp[1]} = \max _{1\leq j\leq n}\sum _{i=1}^n|a_{ij}|$, which is the maximum absolute column sum of the matrix $A$. Using this standard characterization, we obtain the following result.

\begin{theorem}\label{thm - Psi_1-Psi_infty}
The minimal radius of a bounding ball of $\ds$ centered at $D\in\ds$ relative to the operator norm from $\ell^1_n$ to $\ell^1_n$, denoted by $r_1(D)$, is
    $$
    r_1(D) \,=\, 2\Big(1-\min_{i,j} d_{ij}\Big).
    $$
    Moreover, the equality $\textstyle \|D-S\|_{\lp[1]} = 2\left(1-\min_{i,j} d_{ij}\right)$ is realized by $S\in\ds$ if and only if there exist indices $l,k$ such that $s_{lk}=1$ and $\textstyle d_{lk}=\min_{i,j} d_{ij}$.
\end{theorem}
\begin{proof}
    Given $D,S\in\ds$, let $\mathcal{I}(j)$ be the set of indices $i$ such that $d_{ij}-s_{ij}>0$ and $\mathcal{I}(j)^c$ be the set of indices $i$ such that $d_{ij}-s_{ij}\leq0$. Then
    \begin{align*}
        \sum_{i=1}^n|d_{ij} -s_{ij}| 
        \,&=\, \sum_{\mathcal{I}(j)} (d_{ij}-s_{ij}) -\sum_{\mathcal{I}(j)^c} (d_{ij}-s_{ij})\\
        &=\, 2\sum_{\mathcal{I}(j)} (d_{ij}-s_{ij}) - \bigg( \sum_{\mathcal{I}(j)} (d_{ij}-s_{ij}) + \sum_{\mathcal{I}(j)^c} (d_{ij}-s_{ij}) \bigg)\\[-3pt]
        &=\, 2\sum_{\mathcal{I}(j)} (d_{ij}-s_{ij}) -\sum_{i=1}^n (d_{ij}-s_{ij}) \,=\, 2\sum_{\mathcal{I}(j)} (d_{ij}-s_{ij})\\
        &\leq\, 2\sum_{\mathcal{I}(j)} d_{ij},
    \end{align*}
where the last equality is a consequence of the fact that both $D$ and $S$ are doubly stochastic.

Suppose that $\mathcal{I}(j) = \{1,2,\dots,n\}$. On the one hand,  $d_{ij}-s_{ij}>0$ for every $i$. But on the other hand $\sum_{i=1}^n(d_{ij}-s_{ij}) = \sum_{i=1}^n d_{ij} - \sum_{i=1}^n s_{ij} = 0$, a contradiction. Therefore, the cardinality of $\mathcal{I}(j)$ is bounded above by $n-1$. Since $d_{ij}\geq0$ for $1\leq i,j\leq n$, it follows that $\sum_{\mathcal{I}(j)} d_{ij}$ is bounded above by the sum of the $n-1$ largest coefficients $d_{ij}$ (for a fixed $j$). Hence,
\begin{align}\label{eq - eg_psi_1}
    \sum_{i=1}^n|d_{ij} -s_{ij}| \leq 2\sum_{\mathcal{I}(j)} d_{ij} \leq 2\bigg( \sum_{i=1}^n d_{ij} - \min_{i} d_{ij} \bigg) \!= 2\big( 1 - \min_{i} d_{ij} \big).
\end{align}
Taking the maximum over $j$ for $1\leq j\leq n$, we get
\begin{align}
    \|D-S\|_{\lp[1]} &= \max_{j} \sum_{i=1}^n|d_{ij} -s_{ij}|
    \leq \max_{j}\, 2\big(1-\min_{i} d_{ij}\big)\label{eq - D-S}
    = 2\big(1-\min_{i,j} d_{ij}\big).
\end{align}
Taking the maximum on the $S\in\ds$ then yields
\begin{equation}\label{eq - *}
    r_1(D) = \max_{S\in\ds} \|D-S\|_{\lp[1]}
\leq \max_{S\in\ds} 2\big(1-\min_{i,j} d_{ij}\big)
= 2\big(1-\min_{i,j} d_{ij}\big).
\end{equation}
For any matrix $S\in\ds$ with $s_{lk}=1$, where $l$ and $k$ are such that $d_{lk}=\min_{i,j} d_{ij}$, we have $\sum_{i=1}^n |d_{ik}-s_{ik}|=\sum_{i=1,i\neq l}^n d_{ik} + |d_{lk}-1|=2(1-d_{lk})=2\big(1-\min_{i,j} d_{ij}\big)$ and thus by \eqref{eq - *}, $r_1(D) = 2\big(1-\min_{i,j} d_{ij}\big)$. 

 
Now, suppose that $\|D-S\|_{\lp[1]} = r_1(D)$. Then \eqref{eq - D-S} ensures that there exists an index $k$ such that $\sum_{i=1}^n|d_{ik} -s_{ik}| = 2\big(1-\min_{i} d_{ik}\big)$. This identity is satisfied if and only if both inequalities of \eqref{eq - eg_psi_1} are saturated, which is the case if and only if $s_{ik}=0$ for every $i\in \mathcal{I}(k)$ and $\mathcal{I}(k)^c=\{l\}$, where $l$ satisfy $d_{lk}=\min_i d_{ik}$. Hence, we have equality if and only if there exists indices $l,k$ such that $s_{lk}=1$ and $d_{lk}=\min_{i,j} d_{ij}$. 
\end{proof}

\begin{remark}
    Since $\|D-S\|_{\lp[1]} = \|D^* -S^*\|_{\lp[\infty]}$, it directly follows from Theorem \ref{thm - Psi_1-Psi_infty} that $r_\infty(D)$, the minimal radius of a bounding ball of $\ds$ centered at $D\in\ds$ relative to the operator norm from $\ell^\infty_n$ to $\ell^\infty_n$, is equal to
    $
    r_\infty(D) = 2\big(1-\min_{i,j} d_{ij}\big).
    $
    The equality $\textstyle \|D-S\|_{\lp[\infty]} = 2\left(1-\min_{i,j} d_{ij}\right)$ is also realized by $S\in\ds$ if and only if there exist indices $l,k$ such that $s_{lk}=1$ and $\textstyle d_{lk}=\min_{i,j} d_{ij}$.
\end{remark}

\section{The Chebyshev radius and center of the Birkhoff polytope}
\label{sec - Chebyshev}

The \emph{Smallest Enclosing Ball Problem} is a classical question in geometry that generalizes the \emph{Smallest Enclosing Circle Problem} introduced by the 19th century English mathematician mathematician James Joseph Sylvester (who, among many other things, coined the word ``matrix") \cite{sylvester1857question}. Given a metric space $(\mathcal{U},d)$, a nonempty closed constraint set $\mathcal{R}\subseteq \mathcal{U}$ and a nonempty closed, bounded set $\mathcal{B} \subseteq \mathcal{U}$, the Smallest Enclosing Ball Problem lies in finding a point $x \in \mathcal{R}$ and the  smallest radius $r \geq 0$ such that the ball $B(x,r)$ contains the set $\mathcal{B}$, i.e., $\mathcal{B} \subseteq B(x,r)$ and $r$ is optimal. This can be intuitively stated as finding the minimal bounding ball of $\mathcal{B}$ centered at some point in the constraint set $\mathcal{R}$ and relative to the metric space $(\mathcal{U},d)$.

The minimal bounding ball of $\mathcal{B}$ relative to the metric space $(\mathcal{U},d)$ and the constraint $\mathcal{R}$ does not always exist and, if it does, it need not be unique. If it does exist, then the radius $r$ is called the \emph{Chebyshev radius} of $\mathcal{B}$ relative to the metric space $(\mathcal{U},d)$ and the constraint $\mathcal{R}$, denoted $R_d(\mathcal{B})$. As for the point $x \in \mathcal{R}$, it is a \emph{Chebyshev center} of $\mathcal{B}$ relative to the metric space $(\mathcal{U},d)$ and the constraint $\mathcal{R}$. Using this notation, observe that we have 
$$
R_d(\mathcal{B}) \,=\, \infp_{x\in \mathcal{R}} \sup_{y\in \mathcal{B}} d(x,y) \,=\, \inf_{x\in \mathcal{R}} r_d(x).
$$

If $\mathcal{R}=\mathcal{U}$, then the smallest enclosing ball problem is said to be \emph{unconstrained}. In this case, we instead consider the metric space $(\mathcal{R},d)$ and we say that $B(x,r)$ is a minimal bounding ball of $\mathcal{B}$ relative to the metric space $(\mathcal{R},d)$ (likewise for the Chebyshev radius and centers). Moreover, note that in these cases, we need to have $\mathcal{B} \subseteq \mathcal{R}$, as opposed to the general case where the constraint set $\mathcal{R}$ can be entirely unrelated to $\mathcal{B}$.

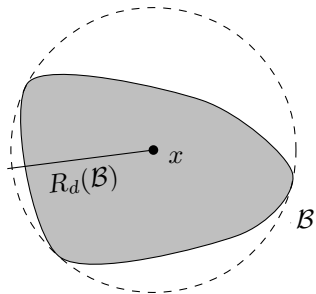
\begin{figure}[ht]
	\pgfmathsetseed{1594627271}
	\centering
	\begin{tikzpicture}[scale=1.6]
		\filldraw[fill=lightgray] plot [smooth cycle, samples=5,domain={1:5}] (\x*360/5+5*rnd:0.5cm+1cm*rnd) node at (0,0) {};
		\fill(1,-0.4)node[below,xshift=.2cm,yshift=.3cm]{$\mathcal{B}$} circle (0cm);
		\draw[dashed](-0.1359,0.205)node[below,xshift=.2cm]{} circle (1.1844);
		\fill(-0.1359,0.205)node[below,xshift=.3cm,yshift=.1cm]{$x$} circle (.04cm);¸
		
		\draw(-0.1359,0.205)--(-1.35,0.05) node[midway,sloped,yshift=-0.3cm]{$R_d(\mathcal{B})$};
		
	\end{tikzpicture}
	\caption{The minimal bounding ball of the nonempty closed, bounded set $\mathcal{B}$ with respect to the metric $d$.}
	\label{fig:centre+rayon}
\end{figure}

More than a century after it has been posed (albeit in a more circumscribed form), the Smallest Enclosing Ball Problem remains an active area of research (see \cite{MR2343160,Mordukhovich2013} and the references therein).

 Before beginning our study of the Chebyshev centers and the Chebyshev radius of the Birkhoff polytope relative to the operator norms from $\ell^p$ to $\ell^p$, it is worthwhile to present some well-known sufficient conditions under which the existence and unicity of the minimal bounding ball of $\mathcal{B}$ is ensured. These will prove useful in the course of our presentation, and in the second part of this series of paper.


\subsection{Results in a general setting}
\label{subsec - gen}

The minimal bounding ball of a set $\mathcal{B}$ does not always exist. For instance, in the 2-dimensional Euclidean space with the constraint set $\mathcal{R}=B(0,1)$, where $B(0,1)$ denote the open unit ball, it is clear that the minimal bounding ball of $\mathcal{B}=\{(2,0)\}$ does not exist. Hence, 
we begin by stating some sufficient conditions, due to Mordukhovich, Nguyen Mau and Villalobos \cite{Mordukhovich2013},  guaranteeing the \textit{existence} of a minimal bounding ball of $\mathcal{B}$.

\begin{proposition}\cite[Theorem 1]{Mordukhovich2013}\label{prop - exist}
    Let $\mathcal{B}$ be a nonempty closed bounded subset in the normed vector space $(V,\|\cdot\|)$. Suppose one of the following holds:
    \begin{enumerate}[label=(\roman*)]
        \item The constraint set $\mathcal{R}\subseteq V$ is nonempty and compact.
        \item The normed vector space $(V,\|\cdot\|)$ is a reflexive Banach space and the constraint set $\mathcal{R}\subseteq V$ is weakly closed.
    \end{enumerate}
    Then there exists a minimal bounding ball of $\mathcal{B}$ relative to $(V,\|\cdot\|)$ and the constraint set $\mathcal{R}$.
\end{proposition}

If a minimal bounding ball of $\mathcal{B}$ exists, its uniqueness is not guaranteed. For example, consider the $\mathbb{R}^2$ with the Euclidean metric. If $\mathcal{R}=\{(x,y): x^2+y^2=1\}$ and $\mathcal{B}=\{(0,0)\}$, then it is clear that a closed unit ball centered at any point of $\mathcal{R}$ is a minimal bounding ball of $\mathcal{B}$. Hence,
we also present a sufficient condition to guarantee the \textit{uniqueness} of the minimal bounding ball of $\mathcal{B}$ in the context of \textit{strictly convex normed vector spaces}, i.e., a normed vector space where $x \neq y$ imply that
\[
\|\lambda x+(1-\lambda)y\| \,<\, \lambda \|x\| + (1-\lambda)\|y\|
\]
for all $0 < \lambda < 1$.


\begin{proposition}\label{prop - uniqueness}
Let $\mathcal{B}$ be a nonempty compact subset in the strictly convex normed vector space $(V,\|\cdot\|)$ and let $\mathcal{R}\subseteq V$ be a nonempty closed, convex constraint set. If there exists a Chebyshev center of $\mathcal{B}$ relative to $(V,\|\cdot\|)$ and the constraint set $\mathcal{R}$, then it is unique.
\end{proposition}
\begin{proof}

Suppose that $x, y \in \mathcal{R}$ are two distinct Chebyshev centers. Then, for every $0 < \lambda < 1$,
    \begin{align*}
        R_{\|\cdot\|}(\mathcal{B}) \,&=\, \infp_{w\in \mathcal{R}} \sup_{z\in \mathcal{B}}\|w-z\| \,\leq\, \sup_{z\in \mathcal{B}}\|\lambda x + (1-\lambda)y-z\| \,=\, \|\lambda x + (1-\lambda)y-z_0\| \\
        &<\, \lambda  \| x-z_0\| + (1-\lambda)  \|y-z_0\| \,\leq\,  \lambda \sup_{z\in \mathcal{B}} \| x-z\| + (1-\lambda) \sup_{z\in \mathcal{B}} \|y-z\| \\[-3pt]
        &=\, R_{\|\cdot\|}(\mathcal{B}),
    \end{align*}
    where $z_0 \in \mathcal{B}$ exists by compactness. This is a contradiction and thus, the Chebyshev center is unique.
\end{proof}

Before moving on to the next result, it is worth recalling what a permutation-invariant subset of $M_n$ is. The definition is reminiscent of that of a permutation-invariant norm presented in Section \ref{PINdef}. Indeed, a set $\mathcal{K} \subset M_n$ is said to be \textit{permutation-invariant} if $PKQ\in\mathcal{K}$ for any $K\in\mathcal{K}$ and any permutation matrices $P, Q \in \Pn$. Observe that the sets $\ds$ and $M_n(\mathbb{R})$ are both permutation-invariant.

For the remainder of this section, we shall focus on proving results about the Chebyshev centers and radius of $\ds$ in the context where the metric is induced by a permutation-invariant norm and where the constraint set is also permutation-invariant. In particular, we shall see that under mild conditions, the special matrix $\alpha J_n$ is a Chebyshev center of $\ds$ for some real number $\alpha$ and that the associated Chebyshev radius is given by $\|\alpha J_n-I_n\|$. To establish this, we first need to prove the following lemma, which is interesting in its own right.

\begin{lemma}\label{lem - bonus1}
	Let $\mathcal{R} \subseteq M_n(\mathbb{R})$ be a permutation-invariant constraint set and let $\|\cdot\|$ be a permutation-invariant norm on $M_n(\mathbb{R})$. Let $A_1, A_2\in\mathcal{R}$ be Chebyshev centers of $\ds$ relative to the metric space $(M_n(\mathbb{R}),\|\cdot\|)$ and the constraint set $\mathcal{R}$. Then,
	\begin{enumerate}[label=(\roman*)]
		\item $PA_iQ$ ($i=1, 2$) is a Chebyshev center of $\ds$ for any permutation matrices $P$ and $Q$;
		\item If the constraint set $\mathcal{R}$ is convex, then any convex combination of $A_1$ and $A_2$ is a Chebyshev center of $\ds$.
	\end{enumerate}
\end{lemma}

\begin{proof}
	We begin by showing $(i)$. Let $A\in\mathcal{R}$. 
	Since $f_A(D):=\|A-D\|$ is a convex real-valued function on $\ds$, it follows from property \ref{prop - conv-max} of doubly stochastic matrices in Section \ref{sec - prop} that the Chebyshev radius is given by
	\[
	R_{\|\cdot \|}(\ds) \,=\, \infp_{A\in \mathcal{R}} \sup_{D\in \ds} \|A-D\| \,=\, \inf_{A\in\mathcal{R}} \max_{P\in \Pn} \|A-P\|.
	\]
    Suppose that $A_0\in \mathcal{R}$ is a Chebyshev center of $\ds$. We thus have $R_{\|\cdot\|}(\ds) = \max_{P\in \Pn} \|A_0-P\|.$ 
	Since $\|\cdot\|$ is permutation-invariant, it follows that for any permutation matrices $Q_1$ and $Q_2$ of appropriate size, we have
	\[
	R_{\|\cdot \|}(\ds) \,=\, \max_{P\in \Pn} \|Q_1A_0Q_2-Q_1PQ_2\| \,=\, \max_{Q\in \Pn} \|Q_1A_0Q_2-Q\|.
	\]
	Hence, $Q_1 A_0 Q_2 \in \mathcal{R}$ is also a Chebyshev center of $\ds$ for any permutation matrices $Q_1$ and $Q_2$.
	
	Now, let us prove statement $(ii)$.
	Since $\mathcal{R}$ is a convex set, $\lambda A_1+(1-\lambda)A_2$ belong to $\mathcal{R}$ for any $0\leq \lambda \leq 1$. Thereby,
	\begin{align*}
		R_{\|\cdot \|}(\ds) \,&=\, \inf_{A\in\mathcal{R}} \max_{P\in \Pn} \|A-P\| \,\leq\, \max_{P\in \Pn} \| \lambda A_1+(1-\lambda)A_2 - P\| \\
		&\leq\, \lambda \max_{P\in \Pn} \|A_1-P\| + (1-\lambda)\max_{P\in\Pn}\|A_2-P\| \\
		&=\, \lambda R_{\|\cdot \|}(\ds) + (1-\lambda)R_{\|\cdot \|}(\ds) \\
		&=\, R_{\|\cdot \|}(\ds),
	\end{align*}
	since $\max_{P\in \Pn} \|A_1-P\|=\max_{P\in \Pn} \|A_2-P\|=R_{\|\cdot \|}(\ds)$. Therefore, in particular,  $R_{\|\cdot \|}(\ds) = \max_{P\in \Pn} \| \lambda A_1+(1-\lambda)A_2 - P\|$. Hence $\lambda A_1+(1-\lambda)A_2 $ is a Chebyshev center of $\ds$ relative the metric space $ (M_n(\mathbb{R}),\|\cdot\|)$ and the constraint set $\mathcal{R}$ for every $0\leq \lambda \leq 1$.
\end{proof}

From this, we derive the following useful theorem.

\begin{theorem}\label{prop - centre J general}
    Let $\mathcal{R} \subseteq M_n(\mathbb{R})$ be a convex permutation-invariant constraint set and let $\|\cdot\|$ be a permutation-invariant norm on $M_n(\mathbb{R})$. If there exist a Chebyshev center $A$ of $\ds$ relative to the metric space $(M_n(\mathbb{R}),\|\cdot\|)$ and the constraint set $\mathcal{R}$, then the matrix $J_nAJ_n =  \big(\frac{1}{n}\sum_{i,j=1}^n a_{ij} \big)J_n$ is also a Chebyshev center of $\ds$ relative to the aforementioned metric space and constraint set. Moreover, the Chebyshev radius of $\ds$ in this setting is given by
    \[
    R_{\|\cdot \|}(\ds) \,=\,\|J_nAJ_n-I_n\| \,=\, \inf_{\!\!\substack{\alpha\in\mathbb{R}\\ \alpha J_n \in \mathcal{R}}\!}\|\alpha J_n-I_n\|
    \vspace{-2pt}
    \]
    and the infimum is attained by $\alpha = \frac{1}{n}\sum_{i,j=1}^n a_{ij}$.
\end{theorem}

\begin{proof}
   Let $A$ be any Chebyshev center in the above settings. Then Lemma \ref{lem - bonus1} 
    implies that the convex combination $ \frac{1}{(n!)^2} \sum_{P,Q\in\Pn} PAQ$ is also a Chebyshev center. Therefore, an application of Lemma \ref{lem - bonus3} yield 
    \[
    \frac{1}{(n!)^2} \!\sum_{P,Q\in\Pn} \!\!PAQ = \left( \frac{1}{n!} \sum_{P\in\Pn} \!P\right)\! A  \left( \frac{1}{n!} \sum_{\smash{Q\in\Pn}} \!Q\right)\! = J_nAJ_n =  \left( \frac{1}{n}\sum_{i,j=1}^n a_{ij} \right)\!J_n.
    \]
   Thus, $J_nAJ_n$ is a Chebyshev center of $\ds$ relative to the metric space $ (M_n(\mathbb{R}),\|\cdot\|)$ and the constraint set $\mathcal{R}$.
    It then follows that $R_{\|\cdot \|}(\ds) = \max_{P\in\Pn} \|J_nAJ_n-P\|$. Consequently, using once again the fact that $\|\cdot\|$ is permutation-invariant and that $J_nP=J_n$ for any permutation matrix $P$, we get
    \[
    R_{\|\cdot \|}(\ds) = \max_{P\in \Pn} \|J_nAJ_nP^*-I_n\| = \max_{P\in \Pn} \|J_nAJ_n-I_n\| = \|J_nAJ_n-I_n\|.
    \]
       Finally, observe that $J_nAJ_n=\alpha J_n$ for some $\alpha\in\mathbb{R}$. Hence,
    \begin{align*}
        R_{\|\cdot \|}(\ds) \,&=\, \inf_{A'\in\mathcal{R}} \max_{P\in\Pn} \|A'-P\| \,\leq \inf_{\substack{\alpha\in\mathbb{R}\\ \alpha J_n \in \mathcal{R}}} \max_{P\in\Pn} \|\alpha J_n-P\| \\
        &= \inf_{\substack{\alpha\in\mathbb{R}\\ \alpha J_n \in \mathcal{R}}}\|\alpha J_n-I_n\| \,\leq\, \|J_nAJ_n-I_n\| \,=\, R_{\|\cdot \|}(\ds)
    \end{align*}
    and the conclusion follows directly.
\end{proof}

As a consequence of Theorem \ref{prop - centre J general}, we derive the following corollary which states that, under mild conditions, any Chebyshev center of $\ds$ is equidistant to every permutation matrix.

\begin{corollary}\label{lem - perm}
Let $\mathcal{R} \subseteq M_n(\mathbb{R})$ be a  convex permutation-invariant constraint set and let $\|\cdot\|$ be a permutation-invariant norm on $M_n(\mathbb{R})$. If there exist a Chebyshev center $A$ of $\ds$ relative to the metric space $(M_n(\mathbb{R}),\|\cdot\|)$ and the constraint set $\mathcal{R}$, then $\|A-P\|=R_{\|\cdot\|}(\ds)$ for any permutation matrix $P\in\Pn$.
\end{corollary}

\begin{proof}
Let $A\in\mathcal{R}$ be a Chebyshev center and let $S\in\Pn$ be any permutation matrix. Observe that, regardless of the choice of $S$, Theorem \ref{prop - centre J general} guarantees that
    \begin{align*}
        R_{\|\cdot\|}(\ds) = \|J_nAJ_n-I_n\| = \|J_nAJ_n-S\| = \bigg\| \frac{1}{(n!)^2} \!\sum_{\smash{P,Q\in\Pn}} \!\!PAQ - S \bigg\|,
    \end{align*}
where the third equality stem from Lemma \ref{lem - bonus3}. Now, Lemma \ref{lem - bonus1} guarantees that $PAQ$ is a Chebyshev center for all $P$ and $Q\in\Pn$ and thus
    \begin{align*}
        R_{\|\cdot\|}(\ds) \,&=\, \bigg\| \frac{1}{(n!)^2} \!\sum_{\smash{P,Q\in\Pn}} \!\!PAQ -S \bigg\| \,=\, \frac{1}{(n!)^2} \bigg\| \sum_{\smash{P,Q\in\Pn}} \!(PAQ - S) \bigg\| \\
        &\leq\, \frac{1}{(n!)^2} \!\sum_{P,Q\in\Pn} \!\!\left\| PAQ - S \right\| \,\leq\, \frac{1}{(n!)^2} \!\sum_{P,Q\in\Pn} \!\max_{R\in\Pn} \left\| PAQ - R \right\| \\
        &=\, \frac{1}{(n!)^2} \!\sum_{P,Q\in\Pn} \!\!R_{\|\cdot\|}(\ds) = R_{\|\cdot\|}(\ds).
    \end{align*}
Hence, all the above inequalities are in fact equalities and, in particular, we find that $\left\| PAQ - S \right\| = \max_{R\in\Pn}\left\| PAQ - R \right\| = R_{\|\cdot\|}(\ds)$ for any $P$ and $Q\in\Pn$. Since $\|\cdot\|$ is permutation-invariant, we conclude that $\|A-P^* SQ^*\|=R_{\|\cdot\|}(\ds)$ for every permutation matrices $P\in\Pn$. Hence $\|A-P\|=R_{\|\cdot\|}(\ds)$ for all $P\in\Pn$.
\end{proof}

Remark that all the above results are valid even if $\mathcal{R}$ has no direct relation to $\ds$. If we also make the natural assumption that $\mathcal{R}=\ds$, then we obtain the following concrete result.

\begin{corollary}\label{cor - centre J}
    If $\|\cdot\|$ is a permutation-invariant norm on $M_n(\mathbb{R})$, then the special doubly stochastic matrix $J_n$ is a Chebyshev center of $\ds$ relative to the metric space $ (\ds,\|\cdot\|)$. Moreover, the associated Chebyshev radius is given by $R_{\|\cdot\|}(\ds)=\|J_n-I_n\|$.
\end{corollary}

\begin{proof}
    Since $\ds$ is compact, Proposition \ref{prop - exist} ensures the existence of a Chebyshev center $D\in\ds$. But Theorem \ref{prop - centre J general} implies that $J_nDJ_n$ is also a Chebyshev center. Since $J_nDJ_n=J_n$ for any $D\in\ds$, it follows that $J_n$ is a Chebyshev center of $\ds$ in these settings and that the Chebyshev radius is equal to $\|J_n-I_n\|$.
\end{proof}


\subsection{The operator norms from \texorpdfstring{$\ell^p_n$}{ℓp} to \texorpdfstring{$\ell^p_n$}{ℓp} for \texorpdfstring{$1 \leq p \leq \infty$}{1≤p≤∞}}
\label{sec - radius2}

With the above tools in hands, we now seek to study the Chebyshev radius $R_p(\ds)$ and the Chebyshev center of $\ds$ relative to the metric space $(\ds,\|\cdot\|_{\lp})$. We begin by deriving a technical lemma which will play a key role in showing the unicity of the Chebyshev center relative to the constraint set $\ds$.

\begin{lemma}\label{lem - A-J_v.p.}
	Consider a matrix $D\in \ds$ with eigenvalues $1=\lambda_1,\lambda_2,\dots,\lambda_n \in \C$. Then the eigenvalues of $D-J_n$ are $0, \lambda_2, \dots, \lambda_n$.
\end{lemma}

\begin{proof}
Observe the following:
\begin{enumerate}[itemindent=20pt,label=(\roman*)]
    \item The matrices $D$ and $J_n$ commute;
    \item The eigenvalue $\lambda_1=1$ of $D$ is associated with the all-ones eigenvector $e$;
    \item The eigenvalues of $J_n$ are $1$ with multiplicity 1 and associated with the eigenvector $e$, and 0 with multiplicity $n-1$.
\end{enumerate}
Theorem 2.3.3 in \cite{HornJohnson2013} implies that there exist a unitary matrix $U$ whose first column is $\tfrac{1}{\sqrt{n}} e$ and such that
\[
U^* D U \,=\, \begin{bmatrix}
    1 & 0 &\cdots & \!0 \\
    * & \!\!\lambda_2\!\! & \cdots & \!0 \\
    \vdots & \vdots & \ddots &\vdots \\
    * & * & \cdots & \!\lambda_n 
\end{bmatrix}
~\quad\& ~\quad
U^* J_n U \,=\, \begin{bmatrix}
    1 & 0 &\cdots & 0 \\
    * & 0 & \cdots & 0 \\
    \vdots & \vdots & \ddots &\vdots \\
    * & * & \cdots & 0 
\end{bmatrix}.
\]
It thus follows that
\[
U^*(D-J_n) U \,=\, \begin{bmatrix}
    1 & 0 &\cdots & \!0 \\
    * & \!\!\lambda_2\!\! & \cdots & \!0 \\
    \vdots & \vdots & \ddots &\vdots \\
    * & * & \cdots & \!\lambda_n 
\end{bmatrix} - \begin{bmatrix}
    1 & 0 &\cdots & 0 \\
    * & 0 & \cdots & 0 \\
    \vdots & \vdots & \ddots &\vdots \\
    * & * & \cdots & 0 
\end{bmatrix} \,=\, \begin{bmatrix}
    0 & 0 &\cdots & \!0 \\
    * & \!\!\lambda_2\!\! & \cdots & \!0 \\
    \vdots & \vdots & \ddots &\vdots \\
    * & * & \cdots & \!\lambda_n 
\end{bmatrix}.
\]
The conclusion directly follows.
\end{proof}

We know from Corollary \ref{cor - centre J} that the Chebyshev radius of $\ds$ relative to the metric space $(\ds,\|\cdot\|_{\lp})$ is given by $\|J_n-I_n\|_{\lp}$. Motivated by this result, the operator norm from $\ell^p_n$ to $\ell^p_n$ of the matrices of the form $\alpha I_n + \beta J_n$ was studied and partial results were presented in \cite{bouthat2}. In particular, it was showed that, 
\begin{align}\label{prop - RIMS}
	\|J_n-I_n\|_{\lp} \,\geqslant\, 1, \quad\quad (1 < p < \infty),
\end{align}
with equality if and only if either $n=2$ or $p=2$. Moreover, in the very special case of $n=2$, we have
\begin{align}\label{n=2}
	\begin{split}
		R_{\|\cdot\|}(\Omega_2) \,&=\, \min_{D\in \ds} \max_{P\in \Pn} \|D-P\| \\&=\, \min_{0\leq a \leq 1} \max\left\{ \left\| \left[\begin{smallmatrix}
			a-1 & 1-a\\
			1-a & a-1
		\end{smallmatrix}\right]
		\right\| , \left\| \left[\begin{smallmatrix}
			a & -a\\
			-a & a
		\end{smallmatrix}\right]
		\right\|  \right\} \\
		&=\, \min_{0\leq a \leq 1} \max\left\{ |1-a| \left\| \left[\begin{smallmatrix}
			1 & -1\\
			-1 & 1
		\end{smallmatrix}\right]
		\right\| , |a|\left\| \left[\begin{smallmatrix}
			1 & -1\\
			-1 & 1
		\end{smallmatrix}\right]
		\right\|  \right\} \\
		&=\, \min_{0\leq a \leq 1} \max\left\{ |1-a| , |a|  \right\}  \left\| \left[\begin{smallmatrix}
			1 & -1\\
			-1 & 1
		\end{smallmatrix}\right]
		\right\| \\
		&=\, \tfrac{1}{2}  \left\| \left[\begin{smallmatrix}
			1 & -1\\
			-1 & 1
		\end{smallmatrix}\right]
		\right\|,
	\end{split}
\end{align}
where $\|\cdot\|$ could be any matrix norm. Since the last equality is satisfied if and only if $a=\tfrac{1}{2}$, we deduce that the special doubly stochastic matrix $J_n$ is the unique Chebyshev center of $\ds$ when $n=2$ relative to any matrix norm when the constraint set is $\ds$.

As for the special case of $p=2$, often refered as the \emph{spectral norm}, recall that we have the alternative definition
\begin{equation}
    \|A\|_{\lp[2]} \,=\, \sigma_1(A),
\end{equation}
where $\sigma_1(A) \geq \sigma_2(A) \geq \dots \geq \sigma_n(A) \in \mathbb{R}$ are the \emph{singular values} of $A$, always arranged in decreasing order. The singular values of $A$ are defined as the square root of the eigenvalues of the matrix $A^*A$. 
Related to the spectral norm and useful in the proof of the following result is the \emph{Frobenius norm}. It is defined as the $2$-norm of the singular values of a matrix, that is
\vspace{-3pt}
\[
\|A\|_\F \,:=\, \left( \sum_{i=1}^n \sigma_i^2(A) \right)^{\!\frac{1}{2}} .
\]

Now, using the previous two results along with Lemma \ref{lem - A-J_v.p.}, we can establish the following theorem.

\begin{theorem}
For $1\leq p \leq \infty$, the special doubly stochastic matrix $J_n$ is the unique Chebyshev center of $\ds$ relative to the metric space $(\ds,\|\cdot\|_{\lp})$. Moreover, for $p=1, \infty$, the Chebyshev radius of $\ds$ is equal to $2\big(1-\tfrac{1}{n}\big)$ and for $p=2$, it is equal to $1$.
\end{theorem}

\begin{proof}
Without loss of generality, since $\|A\|_{\lp[1]} = \|A^*\|_{\lp[\infty]}$, it suffices to consider the cases $1\leq p \leq 2$. We further separate the proof in three parts.

\smallskip
\noindent {\bf Case 1: $\bm{p=1}$.}
First observe that Corollary \ref{cor - centre J} guarantees that $R_1(\ds)=\|J_n-I_n\|_{\lp[1]}$ and that $J_n$ is a Chebyshev center of $\ds$ in these settings. As $\|A\|_{1}=\max_{1\leq j\leq n}\sum_{i=1}^{n}|a_{ij}|$, it is therefore a matter of simple computation to verify that $R_1(\ds)=2\big(1-\tfrac{1}{n}\big)$.
To establish the unicity, let us suppose that $D\in\ds$ is a Chebyshev center of $\ds$ in these settings. Then we have
\[
2\left(1-\tfrac{1}{n}\right) = R_1(\ds) = \max_{\smash{P\in\Pn}} \|D-P\|_{\lp[1]} = r_1(D) = 2\Big(1-\min_{\smash{i,j}} d_{ij}\Big),
\]
where the last equality comes from Theorem \ref{thm - Psi_1-Psi_infty}. Thus, we must have $\min_{i,j} d_{ij}=\frac{1}{n}$ and it then easily follows from the fact that $\sum_{j=1}^n d_{ij}=1$ that this is verified if and only if $D=J_n$.

\smallskip
\noindent {\bf Case 2: $\bm{p=2}$.}
We know from Corollary \ref{cor - centre J} and \eqref{prop - RIMS} that $R_2(\ds)=\|J_n-I_n\|_{\lp[2]}=1$ and that $J_n$ is a Chebyshev center in these settings. To establish the unicity, suppose that $D_0\in\ds$ is a Chebyshev center of $\ds$. Observe that, since $D_0\in \ds$, we have $(D_0-P)^*(D_0-P)e=0$ and thus $0$ is always a singular value of $D_0-P$, for any $P\in\Pn$. Hence, we have 
\begin{align}
    R_\F^2(\ds) \,&=\, \inf_{D\in\ds} \max_{P\in \Pn} \|D-P\|_\F^2\,\leq\, \max_{P\in \Pn} \|D_0-P\|_\F^2  \,=\, \max_{P\in\Pn} \sum_{i=1}^{n-1} \sigma_i^2 (D_0-P) \label{eq - Frob}\\
    &\leq\,  \max_{P\in\Pn} (n-1) \max_{1\leq i \leq n}\sigma_i^2 (D_0-P) \,=\, (n-1) \max_{P\in\Pn}\|D_0-P\|_{\lp[2]}^2 \nonumber\\
    &=\, n-1. \nonumber
\end{align}
Since the Frobenius norm is permutation-invariant and strictly convex \cite{MR4272466}, Corollary \ref{cor - centre J} and Proposition \ref{prop - uniqueness} ensure that $J_n$ is the \emph{unique} Chebyshev center of $\ds$ relative to the Frobenius norm and that $R_\F^2(\ds)=\|J_n-I_n\|_\F^2=n-1$. Hence, the inequalities in \eqref{eq - Frob} are in fact equalities and in particular, we have $\inf_{D\in\ds} \max_{P\in \Pn} \|D-P\|_\F^2 = \max_{P\in \Pn} \|D_0-P\|_\F^2$. Therefore, $D_0$ is a Chebyshev center of $\ds$ relative to the Frobenius norm and it follows that $D_0=J_n$. Hence, $J_n$ is the unique Chebyshev center of $\ds$ relative to the operator norm from $\ell^2$ to $\ell^2$.

\smallskip
\noindent {\bf Case 3: $\bm{1<p<2}$.}
Given a fixed $1<p<2$, let $D\in\ds$ be a Chebyshev center relative to the metric space $(\ds,\|\cdot\|_{\lp})$. Then $\lambda D+ (1-\lambda)J_n \in \ds$ is also a Chebyshev center for any $0\leq \lambda \leq 1$, and thus
\begin{align*}
    R_p(\ds) \,&=\, \|\lambda D+ (1-\lambda)J_n-P\|_{\lp} \\&=\, \max_{\|x\|_p=1} \frac{\|(\lambda D+ (1-\lambda)J_n-P)x\|_p}{\|x\|_p} \\
    &=\, \|(\lambda D+ (1-\lambda)J_n-P)x_0\|_p,
\end{align*}
where $x_0$ denotes an element for which the maximum is realized. But
 \begin{align*}
    \|(\lambda D+ (1-\lambda)J_n-P)x_0\|_p  \,&\leq\, \lambda\|(D-P)x_0\|_p +(1-\lambda)\|(J_n-P)x_0\|_p \\
    &\leq\, \lambda R_p(\ds) + (1-\lambda) R_p(\ds) \\&=\, R_p(\ds).
\end{align*}
As a result, the above inequalities are in fact equalities. We know that the former is saturated if and only if $(D-P)x_0=\alpha (J_n-P)x_0$ for some scalar $\alpha>0$ while the latter is saturated if and only if $\|(D-P)x_0\|_p=R_{p}(\ds)=\|(J_n-P)x_0\|_p$. Combining both conditions, we conclude that $\alpha=1$. It thus follows that $(D-P)x_0=(J_n-P)x_0$, i.e., that $Dx_0=J_nx_0 = \overline{x_0}e$, where $\overline{x}:=\frac{x_1+x_2+\cdots+x_n}{n}$. In particular, $(D-J_n)x_0=0$. Thus, two cases arises: either $D=J_n$, or else $D\neq J_n$ and $x_0$ is an eigenvector of $D-J_n$ associated with a null eigenvalue. In the latter case, Lemma \ref{lem - A-J_v.p.} implies that the eigenvalues of $D-J_n$ are a simple $0$ (associated with the eigenvector $e$ and its scalar multiple) as well as the $n-1$ eigenvalues of $D$ which are not associated with the eigenvector $e$. Hence, we either have $x_0=e/\|e\|_p$ or $J_n x_0=Dx_0=0$. In the first scenario, i.e., if $x_0 = e$, we have
\[
R_p(\ds) \,=\, \|(\lambda D+ (1-\lambda)J_n-I_n)x_0\|_p  \,=\, \frac{\|0\|_p}{\|e\|_p} \,=\, 0,
\]
which is impossible. In the second, namely if $J_n x_0=Dx_0=0$, then
\begin{equation}\label{eq - norm1}
    R_p(\ds) \,=\, \|(\lambda D+ (1-\lambda)J_n-I_n)x_0\|_p  \,=\, \|-x_0\|_p \,=\, 1.
\end{equation}
If $n=2$ or $p=2$, then \eqref{n=2} and Case 2 respectively ensures that $J_n$ is the unique Chebyshev center. But if $n\neq 2$ and $p\neq 2$, then \eqref{prop - RIMS} and \eqref{eq - norm1} imply that $1 = R_p(\ds) = \|J_n-I_n\|_{\lp}>1$, a contradiction. Therefore, it follows that $D=J_n$ and this completes the proof.
%
%
\end{proof}

As we have seen above, for $p=1$ and $\infty$, one can determine the Chebyshev radius of $\ds$ relative to the metric space $(\ds,\|\cdot\|_{\lp})$. However, for every other parameter $p$ (with the notable exception of $p =2$), the question remains open. In fact, we have seen that, since the operator norms from $\ell^p_n$ to $\ell^p_n$ are permutation-invariant, $J_n$ is the unique Chebyshev center of $\ds$ in these settings and that the Chebyshev radius is equal to $\|J_n-I_n\|_{\lp}$. However, the value for this norm is hard to compute for a generic $p$. In recent years, much effort has been put on this problem.

For instance, in  \cite{bouthat1} and \cite{bouthat2}, the following lower and upper bounds for the Chebyshev radius were obtained:
\begin{align}\label{eq - borne}
    1 \,\leqslant\, R_p(\ds) \,\leqslant\, \left(\frac{2(n-1)}{n}\right)^{\left|\frac{2}{p}-1\right|}\!\!, ~\quad (1 < p < \infty).
\end{align}
Observe these bounds coincide if $p = 2$.

In \cite{bouthat2}, it was also proved that, if $n=3$, then
\[
R_p(\Omega_3) = \|J_3-I_3\|_{\lp} = \frac{\left(2^{p-1}+1\right)^{\frac{1}{p}}\big(2^{\frac{1}{p-1}}+1\big)^{1-\frac{1}{p}}}{3}, ~\quad (1< p < \infty).
\]

Moreover, even though we were not able to provide a proof of this result, we proposed the following conjecture.

\begin{conjecture}\label{conj}
    For $1< p < \infty$ and $p\neq 2$, define $\rho:=p-1$. Let $x_p$ be the unique root of the function
    \[
    x ~\longmapsto~ 
    \rho\left(1+x^{\frac{1}{\rho}}\right)\!\left(1-x^{\rho-1}\right)+\left(1+x^{\rho}\right)\!\left(1-x^{\frac{1}{\rho}-1}\right)
    \]
    in the closed interval $[0,1]$. If $m_1:=\big\lfloor \frac{n}{x_{p}+1} \big\rfloor$ and $m_2:=\big\lceil \frac{n}{x_{p}+1} \big\rceil$, then
    \[
    R_p(\ds) = \|J_n-I_n\|_{\lp} \,=  \max_{m \in  \{m_1,m_2\}}\!\! \frac{\left(\left(\frac{n}{m}-1\right)^{p-1}+1\right)^{\!\frac{1}{p}} \!\left(\left(\frac{n}{m}-1\right)^{\frac{1}{p-1}}+1\right)^{\!1-\frac{1}{p}}}{\frac{n}{m}}.
    \]
\end{conjecture}

Note that the above conjecture is in fact a particular case of a broader conjecture stated in \cite{bouthat2}.

\section{Concluding Remarks}

To conclude this paper, we state a few remarks and set out some open questions that appear to be of interest.

\begin{enumerate}[label=(\roman*)]



    \item While the matrix $J_n$ is the Chebyshev center of $\ds$, it is also the \emph{barycenter} of the set of permutation matrices, i.e., it is the arithmetic mean of all the $n\times n$ permutation matrices (see Lemma \ref{lem - bonus3}). Under which conditions the Chebyshev center of the convex hull of a finite set of matrices $\mathcal{M}$ coincide with the barycenter of the set $\mathcal{M}$?

    \item Is Conjecture \ref{conj} valid for every $n\geq4$?
\end{enumerate}

\bibliographystyle{plain}
\bibliography{Part1_Operator}

\end{document}